\documentclass[letterpaper,12pt]{amsart}
\usepackage{graphicx}
\usepackage{amsthm, amsmath, amssymb}
\usepackage[foot]{amsaddr}
\usepackage{fullpage}
\usepackage{mathrsfs,latexsym}
\usepackage{systeme,verbatim,color}
\usepackage{enumitem,marvosym,ifthen}
\usepackage[b]{esvect}
\usepackage[sort,nocompress]{cite}
\usepackage{makecell}
\usepackage{graphicx,tikz}
\usetikzlibrary{graphs,shapes.multipart,shapes.geometric,topaths,calc}
\newtheorem{theorem}{Theorem}[section]

\newtheorem{lem}[theorem]{Lemma}

\newtheorem{prop}[theorem]{Proposition}
\theoremstyle{remark}
\theoremstyle{definition}
\newtheorem{conj}[theorem]{Conjecture}

\numberwithin{equation}{section}
\newcommand{\abs}[1]{\left\lvert#1\right\rvert}
\newcommand{\set}[1]{\left\lbrace#1\right\rbrace}
\newcommand{\setst}[2]{\left\lbrace#1\ \middle|\ #2\right\rbrace}

\newcommand{\ab}{\allowbreak}
\DeclareMathOperator{\ex}{\mathit{\lambda}}

\DeclareMathOperator{\exl}{\mathit{\lambda}_{\rm local}}

\DeclareMathOperator{\exlin}{\mathit{\lambda}_{\rm local(in)}}

\DeclareMathOperator{\pil}{\mathit{\lambda}_{\rm local}}
\DeclareMathOperator{\pili}{\mathit{\lambda}_{\rm local(in)}}
\DeclareMathOperator{\pilin}{\mathit{\lambda}_{\rm local(in)}}

\DeclareMathOperator{\Gin}{\mathit{G}_{v({\rm in})}}
\DeclareMathOperator{\gin}{\mathit{g}_{v({\rm in})}}

\DeclareMathOperator{\pilout}{\mathit{\lambda}_{\rm local(out)}}
\DeclareMathOperator{\pilo}{\mathit{\lambda}_{\rm local(out)}}
\DeclareMathOperator{\Gout}{\mathit{G}_{v({\rm out})}}
\DeclareMathOperator{\gout}{\mathit{g}_{v({\rm out})}}

\DeclareMathOperator{\wt}{wt}
\DeclareMathOperator{\density}{d}

\DeclareMathOperator{\degree}{d}
\newcommand{\M}{\ensuremath{\mathscr{M}}}

\newcommand{\floor}[1]{\lfloor#1\rfloor}
\newcommand{\ceil}[1]{\left\lceil#1\right\rceil}
\newcommand{\ds}{\displaystyle}

\begin{document}
\title{Inducibility in the hypercube}
\author{John Goldwasser$^\dagger$}
\address{$^\dagger$West Virginia University}
\email{jgoldwas@math.wvu.edu}
\author{Ryan Hansen$^\dagger$}
\email{rhansen@math.wvu.edu}
%
\begin{abstract}{
Let $Q_d$ be the hypercube of dimension $d$ and let $H$ and $K$ be subsets of the vertex set $V(Q_d)$, called configurations in $Q_d$.  We say that $K$ is an \emph{exact copy} of $H$ if there is an automorphism of $Q_d$ which sends $H$ onto $K$.  Let $n\geq d$ be an integer, let $H$ be a configuration in $Q_d$ and let $S$ be a configuration in $Q_n$.  We let $\ex(H,d,n)$ be the maximum, over all configurations $S$ in $Q_n$, of the fraction of sub-$d$-cubes $R$ of $Q_n$ in which $S\cap R$ is an exact copy of $H$, and we define the $d$-cube density $\lambda(H,d)$ of $H$ to be the limit as $n$ goes to infinity of $\ex(H,d,n)$.  We determine $\lambda(H,d)$ for several configurations in $Q_3$ and $Q_4$ as well as for an infinite family of configurations.  There are strong connections with the inducibility of graphs.
}
\end{abstract}
\maketitle
\section{Introduction} 
\label{sec:introduction}

	The $n$-cube, which we denote by $Q_n$, is the graph whose vertex set $V_n=V(Q_n)$ is the set of all binary $n$-tuples, with two vertices adjacent if and only if they differ in precisely one coordinate (so Hamming distance 1). Let $[n]=\set{1,2,\ldots,n}$.  We sometimes denote a vertex $(x_1,x_2,\ldots,x_n)$ of $Q_n$ by the subset $S$ of $[n]$ such that $i\in S$ if and only if $x_i=1$. So if $n=4$, then $\emptyset$ denotes $( 0 0 0 0)$, and $\set{1,3}$ or 13, denotes $( 1 0 1 0)$ and $\set{\set{1},\set{1,3}}$ (or $\set{1,13}$) denotes $\set{( 1 0 0 0 ), ( 1 0 1 0 )}$.  The weight of a vertex is the number of 1s.  For each positive integer $d$ less than or equal to $n$, $Q_n$ has $\binom{n}{d}2^{n-d}$ subgraphs which are isomorphic to $Q_d$ ($d$ coordinates can vary, while $n-d$ coordinates are fixed).  We sometimes refer to each of the $d$ coordinates that can vary as a \emph{flip bit}.  A single flip bit determines an edge while a set of $d$ flip bits, with the other $n-d$ coordinates fixed, determines a particular sub-$d$-cube.
	
	Let $H$ and $K$ be subsets of $V_d$ (we call $H$ and $K$ configurations in $Q_d$).  We say $K$ is an \emph{exact copy} of $H$ if there is an automorphism of $Q_d$ which sends $H$ onto $K$.  For example, $\set{\emptyset,12}$ is an exact copy of $\set{2,123}$ in $Q_3$, but $\set{2,13}$ is not (the vertices are distance 3 apart). So if $K$ is an exact copy of $H$ then they induce isomorphic subgraphs of $Q_d$, but the converse may not hold.
	
	Let $d$ and $n$ be positive integers with $d\leq n$, let $H$ be a configuration in $Q_d$ and let $S$ be a configuration in $Q_n$.  We let $G(H,d,n,S)$ denote the number of sub-$d$-cubes $R$ of $Q_n$ in which $S\cap R$ is an exact copy of $H$ and define $g(H,d,n,S)$ by
	\[
		g(H,d,n,s) =\frac{G(H,d,n,S)}{\binom{n}{d}2^{n-d}}
	\]
	So $g(H,d,n,S)$ is the fraction of all sub-$d$-cubes whose intersection with $S$ is an exact copy of $H$.  We define $\ex(H,d,n)$ to be the maximum of $g(H,d,n,S)$ over all configurations $S$ in $Q_n$.  Note that $\lambda(H,d,n)$ is the average of $2n$ densities $g(H,d,n-1,S_j)$, each of them the fraction of sub-$d$-cubes $R$ in a sub-$(n-1)$-cube of $Q_n$ in which $R\cap S_j$ is an exact copy of $H$, where $S_j$ is the intersection of a maximizing configuration in $Q_n$ with one of the $2n$ sub-$(n-1)$-cubes.  Hence $\ex(H,d,n)$ is the average of $2n$ densities, each of them less than or equal to $\ex(H,d,n-1)$, which means $\ex(H,d,n)$ is a nonincreasing function of $n$, so we can define the $d$-cube density $\lambda(H,d)$ of $H$ by
	\[
		\ex(H,d)=\lim_{n\to\infty} \ex(H,d,n).
	\]
So $\lambda(H,d)$ is the limit as $n$ goes to infinity of the maximum fraction, over all $S\subseteq V_n$, of ``good'' sub-$d$-cubes -- those whose intersection with $S$ is an exact copy of $H$.

In \cite{perfectpathpaper} we initiated the investigation of $d$-cube-density.  There have been many papers on Tur\'{a}n and Ramsey type problems in the hypercube.  There has been extensive research on the minimum fraction of edges in $Q_n$ one can take with no cycle of various lengths \cite{Chung:1992,Conlon:2010,Furedi:2015,Rahman} and a few papers on vertex Tur\'{a}n problems in $Q_n$ \cite{kostochka1976piercing,Johnson:2010,Johnson:1989cy}. There has also been extensive work on which monochromatic cycles must appear in any edge coloring of a large hypercube with a fixed number of colors \cite{Alon:2006,Axenovich:2006,Chung:1992,Conder:1993} and a few resultls on which vertex structures must appear \cite{Goldwasser:2012}.  In \cite{AKS:2006} and \cite{group_paper}, results were obtained on \emph{polychromatic colorings} of $Q_n$.  These are edge colorings with $p$ colors such that every sub-$d$-cube contains every color.

We wanted to investigate a different extremal probelm in the hypercube: the maximum density of a certain kind of sub-structure.  Instead of using graph isomorphism to determine if two substructures are the same, it seemed to capture the essence of the $n$-cube better if the small structures were ``rigid'' within a sub-$d$-cube, and that is what motivated our definitions of exact copy of a configuration and $d$-cube density.  While $d$-cube density is not the same thing as graph inducibility, there are similarities between the two notions and, in fact, we use several results on inducibility (see Section \ref{sec:inducibility}) in our proofs.

In \cite{perfectpathpaper} we used a kind of \emph{blow--up} to show that $\lambda(H,d)\geq\frac{d!}{d^d}$ for every configuration $H$ in $Q_d$.  We defined a perfect $2d$-cycle $C_{2d}$ in $Q_d$ to be a cycle with $d$ pairs of vertices each Hamming distance $d$ apart.  We showed that $\lambda(C_8,4)=\frac{4!}{4^4}$, achieving the smallest possible value for any configuration in $Q_4$.  We also showed $\lambda(P_4,3)=\frac{3}{8}$ where $P_4$ is the induced path in $Q_3$ with 4 vertices.

Finding $d$-cube density seems to be very difficult even for most small configurations.  In this paper we find the $d$-cube density for three configurations in $Q_3$, two configurations in $Q_4$, and for an infinite family of configurations, $d-1$ of them in $Q_d$ for each $d\geq 2$.  We find a construction to produce a lower bound and then find a matching upper bound by using known results on the inducibility of small graphs to show the local density cannot be larger. For all configurations in $Q_2$ and $Q_3$ for which we could not determine the $d$-cube density, we give a construction to produce a lower bound and give an upper bound computed by Rahil Baber \cite{Baber:2014p} using flag algebras.

In Section \ref{sec:local_d_cube_density} we define local $d$-cube density, the notion we use to find the upper bounds.  In Section \ref{sec:configurations_in_q_2} we consider the possible configurations in $Q_2$.  In Section \ref{sec:inducibility} we summarize the results on inducibility of graphs which we will need.  In Section \ref{sec:layered_configurations} we discuss \emph{layered} configurations in $Q_d$, those that are defined in terms of the weights of the $d$-vectors.  In Section \ref{sec:an_infinite_family} we find $d$-cube density for a nontrivial infinite family of configurations.  In Section \ref{sec:configurations_in_q_3_} we consider $d$-cube density for configurations in $Q_3$, and in Section \ref{sec:configurations_in_q_4} we determine the $d$-cube density for two configurations in $Q_4$.  In Section \ref{sec:open_problems_and_conjectures} we discuss some open problems and state some conjectures.


\section{Local $d$-cube density} 
\label{sec:local_d_cube_density}

	Let $H$ be a configuration in $Q_d$ and $S$ be a a configuration in $Q_n$.  For each vertex $v\in S$, we let $\Gin(H,d,n,S)$ be the number of sub-$d$-cubes $R$ of $Q_n$ containing $v$ in which $S\cap R$ is an exact copy of $H$ and $\gin(H,d,n,S)=\frac{\Gin(H,d,n,S)}{\binom{n}{d}}$.  So $\gin(H,d,n,S)$ is the fraction of sub-$d$-cubes containing $v$ which have an exact copy of $H$.  We define $\pilin(H,d,n)$ to be the maximum of $\gin(H,d,n,S)$ over all $v$ and $S$ where $v\in S\subseteq V_n$.  As with $\ex(H,d,n)$, a simple averaging argument shows that $\pilin(H,d,n)$ is a nonincreasing function of $n$, so we define $\pilin(H,d)$ by
	\[
		\pilin(H,d)=\lim_{n \to \infty}\exlin(H,d,n)
	\]
	For each $v\not\in S$, we perform a similar procedure to define $\Gout(H,d,n,S)$, $\gout(H,d,n,S)$, and $\pilout(H,d)$.  So $\pilin(H,d)$ and $\pilout(H,d)$ are the maximum local densities of sub-$d$-cubes with an exact copy of $H$ among all sub-$d$-cubes containing $v$ in $S$ and out of $S$ respectively.  Finally, we define $\pil(H,d)$ to be $\max\{\pilin(H,d),\pilout(H,d)\}$.  Since the global density cannot be more than the maximum local density, we must have $\lambda(H,d)\leq \pil(H,d)$.
	
	We refer to the following type of construction as a \emph{partition-modular construction}.  These are constructions generated by choosing a partition of $[n] = A_1\cup A_2\cup\cdots\cup A_i$ and taking as $S$ the set of vertices such that their binary $n$-tuples satisfy a chosen set of congruences for the weight of the coordinates within the parts.  Sometimes we denote this as a list of $i$-tuples along with a list of their moduli for convenience. For example, the partition $A_1\cup A_2$ taking 01 mod $(2,2)$ would indicate the partition $[n]= A_1\cup A_2$ with $S$ being all vertices with weight 0 mod 2 in $A_1$ and weight 1 mod 2 in $A_2$.  The fractional sizes of the $A_i$ which maximize the number of $Q_d$s having the configuration may also be indicated.
	
	Note that the sets in the partition may be of any sizes, however, when $i=1$ we call such a configuration \emph{layered} since it is equivalent to choosing all vertices of particular weights modulo $a$ (i.e. entire ``levels'' of $Q_n$).


\section{Configurations in $Q_2$} 
\label{sec:configurations_in_q_2}

	It is obvious that $\lambda(H,d)=\lambda\left(\overline{H},d\right)$ where $\overline{H}$ is the complement of the configuration $H$ in $Q_d$, so we may restrict our consideration to only one configuration in each of the complementary pairs.

	Sketches of all of the configurations in $Q_2$, subject to the above restriction, are given in Figure \ref{fig:Q2}.  In the figure, red vertices are in the configuration and open blue are not.  The edges have been added for emphasis, but configurations are sets of vertices.

\begin{figure}[htbp]
	\centering
	\begin{tikzpicture}[scale=4,inner sep=.85mm]
		\foreach \i in {1,2,...,4} {
		\node (a\i) at (\i,0) {};
		\node (b\i) at (\i+.5,0) {};
		\node (c\i) at (\i+.5,.5) {};
		\node (d\i) at (\i,.5) {};
		\node at (\i+.25,-.25) {$Z_\i$};
		}
		\foreach \i in {1,...,4} {
			\draw [dashed, blue] (a\i) rectangle (c\i);
		}
		\foreach \i in {1,...,4} {
		\foreach \j in {a,b,c,d} {
			\filldraw[thick,draw=blue,fill=white] (\j\i) circle (.15mm);
		}}
		\foreach \i in {2,3} {
		\node[circle,fill=red] () at (\i,0) {};
		}
		\node[circle,fill=red] () at (4,.5) {};
		\node[circle,fill=red] () at (4.5,0) {};
		\node[circle,fill=red] () at (3,.5) {};
		
		\draw [ultra thick,red] (a3) -- (d3);
	\end{tikzpicture}
	\caption{Configurations in $Q_2$.}
	\label{fig:Q2}
\end{figure}
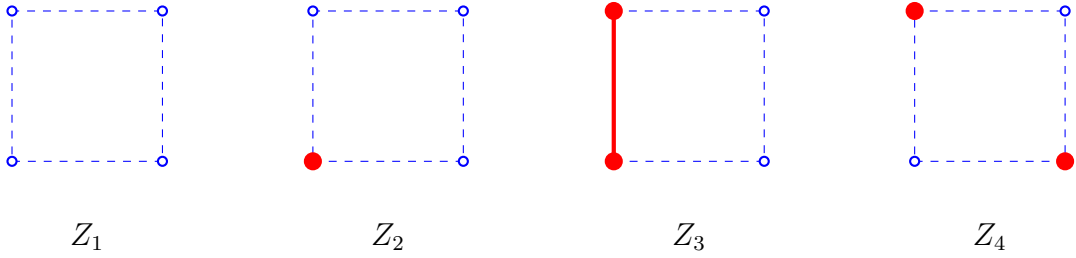

\subsection{Lower Bounds by Construction} 
\label{subsec:lower_bounds_by_construction}

Clearly $\lambda(Z_1,2)=1$, since we would simply take $S=\emptyset$.  To show $\lambda(Z_4)=1$, we take $S$ to be the layered configuration $0\mod 2$ in $Q_n$ (all vertices with even weight).

We can find a lower bound for $Z_2$ by considering the layered configuration in $Q_n$ given by $0\mod 3$.  This gives $\frac{2}{3}\leq \lambda(Z_2,2)$ since any sub-2-cube with the smallest weight of a vertex congruent to 1 or 2 mod 3 has an exact copy of $Z_2$.  For an upper bound we have Baber's flag algebra computation $.685714$.  We remark that this is close to $\frac{24}{35}$, though this may reflect the flag algebra construction rather than the actual value.

A construction for $Z_3$ is given by considering $[n]=A\cup B$ and then taking $S$ to be the set of all vertices given by binary $n$-tuples with weight $0\mod 2$ in $A$.  This gives a ``good'' $Q_2$ for each $Q_2$ with one flip bit in each of $A$ and $B$.  If we take $\abs{A}=\floor{\frac{n}{2}}$ and $\abs{B}=\ceil{\frac{n}{2}}$, this results in $\floor{\frac{n}{2}}\ceil{\frac{n}{2}}\approx \frac{n^2}{4}$ many ``good'' $Q_2$s, which shows $\frac{1}{2}\leq \lambda(Z_3,2)$.

Table \ref{bestQ2s} summarizes the best results obtained in $Q_2$.
\begin{table}[htbp]
	\centering
		\begin{tabular}{|c|c|c|c|}
			\hline
		Configuration & Construction & Lower Bound & Upper Bound\\
			\hline\hline
		$Z_1$ & $\emptyset$ & 1 & 1\\
		\hline
		$Z_2$ & Layered: $0$ mod $3$ & $2/3$ & $.685714286$\\
		\hline
		$Z_3$ & $A\cup B$ taking $0$ mod $2$ in $A$ & $1/2$ & $1/2$ (Theorem \ref{thm:V_3_in_Q_2})\\
		\hline
		$Z_4$ & Layered: $0$ mod $2$ & 1 & 1\\
		\hline

		\end{tabular}
	\caption{Summary of the best results for configurations in $Q_2$.}
	\label{bestQ2s}
\end{table}


\subsection{Upper Bounds} 
\label{sec:upper_bounds_q_2}

In order to show that $\lambda(Z_3,2)=\frac{1}{2}$, we use an argument that will be applied, in a slightly more general form, to an infinite family of configurations in Section \ref{sec:an_infinite_family}.

\begin{theorem}\label{thm:V_3_in_Q_2}
		$\lambda(Z_3,2)=\frac{1}{2}$.
	\begin{proof}
		We showed $\lambda(Z_3,2)\geq \frac{1}{2}$ in Section \ref{subsec:lower_bounds_by_construction}. Let $v$ be a vertex in $S$ and let $S_v=S\cap N(v)$, where $N(v)$ is the neighborhood of $v$ in $Q_n$.  A $Q_2$ containing $v$ can be ``good'' only if one of the two vertices adjacent to $v$ is in $S_v$ and one is not.  Hence
		\[
			\gin(Z_3,2,n,S)\leq \frac{\abs{S_v}\left(n-\abs{S_v}\right)}{\binom{n}{2}}\leq \frac{\ceil{\frac{n}{2}}\floor{\frac{n}{2}}}{\binom{n}{2}}
		\]
		and
		\[
			\pilin(Z_3,2)\leq \lim_{n \to \infty}\frac{\ceil{\frac{n}{2}}\floor{\frac{n}{2}}}{\binom{n}{2}} = \frac{1}{2}.
		\]
		
		Since $Z_3$ is self-complementary in $Q_2$, $\frac{1}{2} \leq \lambda(Z_3,2) \leq \pilout(Z_3,2)\ab = \pilin(Z_3,2)\ab = \exl(Z_3,2) \leq \frac{1}{2}$.
	\end{proof}
\end{theorem}


\section{Inducibility} 
\label{sec:inducibility}

There are strong connections between $d$-cube density and \emph{inducibility} of a graph.  Given graphs $G$ and $H$, with $\abs{V(G)}=n$ and $\abs{V(H)}=k$, the \emph{density} of $H$ in $G$, denoted $\density_H(G)$, is defined by
\[
	\density_H(G)=\frac{\textrm{\# induced copies of $H$ in $G$}}{\binom{n}{k}}
\]
Pippinger and Golumbic \cite{PG:1975} defined the inducibility $i(H)$ of $H$ by
\[
	i(H)=\lim_{n \to \infty} \max_{\abs{V(G)}=n} \degree_H(G)
\]
Clearly $i(H)=i(\overline{H})$ where $\overline{H}$ is the complement of $H$.  We summarize a few inducibility results, some of which we will use to prove upper bounds for $d$-cube density.
\begin{enumerate}
	\item\label{K12} $i(K_{1,2})=\frac{3}{4}$.  The optimizing graph $G$ is $K_{\frac{n}{2},\frac{n}{2}}$.  That it cannot be larger than $\frac{3}{4}$ follows immediately from a theorem of Goodman \cite{G:1985} that says that in any 2-coloring of the edges of $K_n$, at least $\frac{1}{4}$ (asymptotically) of the $K_3$s are monochromatic.
	\item $i(K_{2,2})=\frac{3}{8}$.  In \cite{Bollobas:1986bfa}, Bollob\'{a}s et. al. showed that the graph on $n$ vertices which has the most induced copies of $K_{r,r}$, for any $r\geq 2$, is $K_{\left\lceil \frac{n}{2}\right\rceil,\left\lfloor\frac{n}{2}\right\rceil}$.
	\item Since $2K_2$ is the complement of $K_{2,2}$, $i(2K_2)=\frac{3}{8}$.  In \cite{perfectpathpaper}, Goldwasser and Hansen showed that the bipartite graph on $n$ vertices with the most induced copies of $2K_2$ is two disjoint copies of $K_{\frac{n}{4},\frac{n}{4}}$ (with the obvious modification if $n$ is not divisible by 4).  This graph has $\frac{3}{32}$ of its subgraphs induced by 4 vertices isomorphic to $2K_2$.  They needed this result to show that $\ex(C_8,4)=\frac{3}{32}$ where $C_8$ is the ``perfect'' 8-cycle.
	\item\label{K13} $i(K_{1,3})=\frac{1}{2}$.  In \cite{BS:1994}, Brown and Siderenko showed that the graph on $n$ vertices which has the most induced copies of $K_{r,s}$, for any $r,s$ (except $r=s=1$), is complete bipartite.  The optimizing graph for $K_{1,3}$ is asymptotically, but not exactly, equibipartite; the sizes of the parts are roughly $\frac{n}{2}\pm\frac{\sqrt{3n}}{2}$.
	\item In \cite{Hirst:2014}, Hirst used flag algebras to show that $i(K_{1,1,2})=\frac{72}{125}=.576$ and $i(K_{\rm{PAW}})=\frac{3}{8}$ where $K_{1,1,2}$ is $K_4$ minus an edge and $K_{\rm{PAW}}$ is $K_3$ plus a pendant edge, leaving the path $P_4$ as the only graph on 4 vertices whose exact inducibility has yet to be determined.
	\item In \cite{EL:2015}, Even-Zohar and Linial improve earlier best bounds \cite{Exoo:1986,V:2013} for $i(P_4)$ and find the inducibility of some graphs on 5 vertices.
\end{enumerate}


\section{Layered Configurations} 
\label{sec:layered_configurations}

Recall that we say a configuration $H$ in $Q_d$ is \emph{layered} if it is an exact copy of a configuration $K$ in $Q_d$ such that $v\in K$ if and only if $\wt(v)\in W_H$ for some subset $W_H$ of $[0,d]$.  For example, $H=\{1001,\ab 1110,\ab 0010,\ab 0100,\ab 0111\}$ is layered because there is an automorphism of $Q_4$ (interchange 0 and 1 in the 2$^{\rm{nd}}$ and 3$^{\rm{rd}}$ coordinates) which maps $H$ onto $K=\{1111,\ab 1000,\ab 0100,\ab 0010,\ab 0001\}$and $K=\{v\in Q_4 : \wt(v)=1\textrm{ or }4\}$.  We call $K$ a \emph{canonical layered configuration}.  

If $W_H$ is the set of weights for a canonical layered configuration $H$ in $Q_d$, and if $j+2\in W_H$ if and only if $j\in W_H$ for all $j\in[0,d-2]$, then $H$ is either all vertices, none of the vertices, all even weight vertices, or all odd weight vertices in $Q_d$.  We call these four types of configurations trivial. Each trivial configuration has $d$-cube density equal to 1 (using the obvious layered configuration $S$ in $Q_n$).  The configurations $W_1, W_3, W_7, W_8, W_{12}$ and $W_{14}$ (and their complements) are layered configurations in $Q_3$, with $W_1$ and $W_{14}$ trivial (See Figure \ref{fig:Q3} and Table \ref{bestQ3s}).

One can get a good lower bound construction for any layered configuration in $Q_d$ by using an appropriate layered configuration $S$ in $Q_n$. For example, if we represent the configuration $W_8$ by $H=\{110,101,011\}$, we define $S$ by $S=\{v\in V_n : \wt(v)\equiv 2\mod 3\}$.  Any sub-3-cube of $Q_n$ whose smallest weight vertex has weight congruent to 0 or 1 mod 3 is ``good'', showing that $\lambda(W_8,3)\geq\frac{2}{3}$.  Baber's flag algebra upper bound is $.66666666675$ so undoubtedly $\lambda(W_8,3)=\frac{2}{3}$, but we have not proved it.


\begin{theorem}\label{thm:local_trivial_layered}
	If $H$ is a configuration in $Q_d$, then $\pil(H,d)=1$ if and only if $H$ is layered.
\end{theorem}

To prove Theorem \ref{thm:local_trivial_layered} we will need the following lemma which can be proved using a ``tower'' of applications of Ramsey's theorem.

\begin{lem}\label{lem:monochromatic_Q_d}
	\cite{Goldwasser:2012} For all positive integers $k$ and $d$, there exists a positive integer $N=N(k,d)$ such that for all $n\geq N$, if $c$ is any vertex coloring of $Q_n$ with $k$ colors, for each $v\in V_n$, there exists a sub-$d$-cube $D$ containing $v$ such that the coloring induced by $c$ on $D$ is layered with ``bottom'' vertex $v$ (so $v$ plays the role of $\emptyset$ is a canonical layered coloring).
\end{lem}

\begin{proof}[Proof of Theorem \ref{thm:local_trivial_layered}]
	Suppose $H$ is layered with weight set $W_H\subseteq [0,d]$.  For each $n\geq d$, if $S$ is any layered configuration in $Q_n$ with weight set $W_S$ such that $W_H\subseteq W_S$, then every sub-$d$-cube containing $\emptyset$ has an exact copy of $H$, so $\pil(H,d)=1$.
	
	Conversely, suppose $\pil(H,d)=1$. We can view a configuration $S$ in $Q_n$ as a $2$-coloring of $V_n$ -- red if in $S$, blue if not in $S$.  Hence, by Lemma \ref{lem:monochromatic_Q_d}, there exists an integer $N$ such that for all $n\geq N$ and all $v\in V_n$, if $S$ is any configuration in $Q_n$ then there exists a sub-$d$-cube $D$ containing $v$ such that $D\cap S$ is a layered configuration in $D$.  Since $\pil(H,d,n)$ is a nonincreasing function of $n$, and since every vertex $v$ is contained in a sub-$d$-cube which has an exact copy of a layered configuration, if $H$ is not layered then $\pil(H,d)<1$.
\end{proof}

Since $\pil(H,d)=1$ for every layered configuration $H$ in $Q_d$, our usual procedure of using $\pil(H,d)$ to get an upper bound for $d$-cube density cannot work and that is why we have not been able to obtain upper bounds by hand for any nontrivial layered configuration in $Q_d$ for any $d$. As with $W_8$, Baber's flag algebra upper bounds for $\ex(W_7,3)$ and $\ex(W_{12},3)$ are within $10^{-9}$ of our lower bounds ($\frac{1}{3}\leq \ex(W_7,3)$ and $\frac{1}{2}\leq \ex(W_{12},3)$). The flag algebra upper bound for $\ex(W_3,3)$ (one vertex in $Q_3$) is $.610043$, not so close to our lower bound of $\frac{1}{2}$.

\begin{conj}
	$\ex(W_7,3)=\frac{1}{3}, \ex(W_8,3)=\frac{2}{3},\ex(W_{12},3)=\frac{1}{2}$.
\end{conj}


\section{An Infinite Family} 
\label{sec:an_infinite_family}

Theorem \ref{thm:V_3_in_Q_2} can be generalized to apply to an infinite family of configurations containing $Z_3$, $W_{6}$, and $W_{13}$.  Let $d$ and $i$ be positive integers with $1\leq i<d$. We define the configuration $H(d,i)$ in $Q_d$ by
\begin{align*}
	H(d,i) &=\setst{(v_1,v_2,\ldots,v_d)\in V_d}{\sum_{j=1}^{i}v_j \textrm{ is even}}.
\end{align*}

\begin{theorem}\label{regularfamily}
	$\lambda(H(d,i),d)=\binom{d}{i}\frac{i^i (d-i)^{d-i}}{d^d}$.
	\begin{proof}
		Each vertex in $H(d,i)$ has precisely $d-i$ neighbors in $H(d,i)$ (change any one of the last $d-i$ coordinates).  Since $H(d,i)$ is self-complementary in $Q_d$ $\left(\sum_{j=1}^i v_j \textrm{ is odd} \right)$, $\pil(H(d,i),d)=\pilin(H(d,i),d)=\pilout(H(d,i),d)$.
		
		If $n\geq d$ and $v\in S\subseteq V_n$ and $R$ is a sub-$d$-cube of $Q_n$ containing $v$, then $R$ can be good only if precisely $d-i$ neighbors of $v$ in $R$ are in $S$.  If $x$ is the fraction of neighbors of $v$ in $Q_n$ which are in $S$, then the fraction of sub-$d$-cubes of $Q_n$ containing $v$ which have precisely $d-i$ neighbors in $S$ is $f(x)=\binom{d}{i}x^{d-i}(1-x)^i$.  By simple calculus, $f(x)$ is maximized on $[0,1]$ when $x=\frac{d-i}{d}$, so $\pilin(H(d,i))\leq\binom{d}{i}\frac{(d-i)^{d-i} i^i}{d^d}$.
		
		To show this upper bound is a lower bound as well, let $S=${\large\{}$(x_1,x_2,\ldots,x_n):\ab\sum_{j=1}^m v_j$ is even, where $m=\floor{\frac{in}{d}}${\large\}}.  Then any sub-$d$-cube of $Q_n$ with precisely $i$ flip bits in  $[1,m]$ is good, and this is a fraction
		\[
			\binom{d}{i}\frac{m^i(n-m)^{d-i}}{n^d} = \binom{d}{i}\left( \frac{\floor{\frac{in}{d}}}{n} \right)^i \left( \frac{\ceil{\frac{(d-i)n}{d}}}{n} \right)^{d-i}
		\]
		of all sub-$d$-cubes, and the limit as $n$ goes to infinity is $\binom{d}{i}\frac{i^i(d-i)^{d-i}}{d^d}$.
	\end{proof}
\end{theorem}

Note that the configuration $W_6$ in $Q_3$ is $H(3,1)$, $W_{13}$ is $H(3,2)$, and $Z_3$ in $Q_2$ is $H(2,1)$.  Further note that $\lim_{d\to\infty}\lambda(H(d,i),d)=\frac{i^i}{i!e^i}$.  In particular, when $i=1$, $\lim_{d\to \infty} (H(d,1),d)=\frac{1}{e}$. ($H(d,1)$ is a copy of $Q_{d-1}$ in $Q_d$)


\section{Configurations in $Q_3$} 
\label{sec:configurations_in_q_3_}

The configurations in $Q_3$ are illustrated in Figure \ref{fig:Q3} and listed in Table \ref{bestQ3s}. Recalling that $\lambda(H)=\lambda\left(\overline{H}\right)$ where $\overline{H}$ is the complement of $H$ in $Q_d$, the table lists just one configuration in each complementary pair.

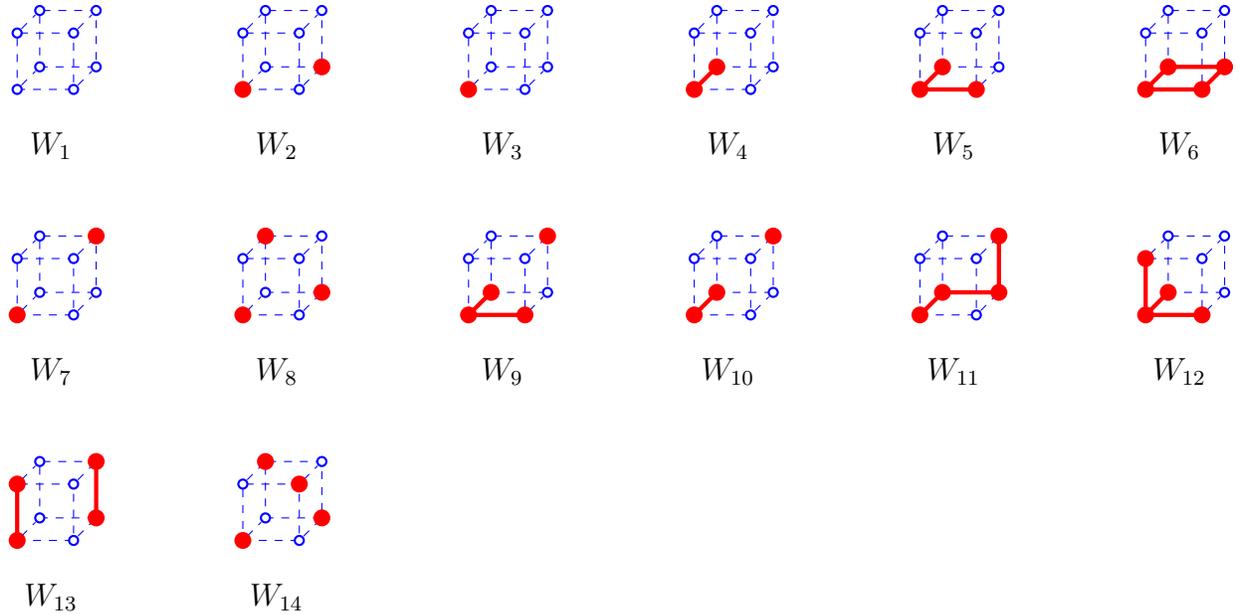
\begin{figure}[htbp]
	\centering
	\begin{tikzpicture}[scale=1.5,inner sep=.5mm]
		\foreach \j in {0,1,2} {
		\foreach \i in {1,...,6} {
		\pgfmathsetmacro{\subscriptlabel}{int(\i+6*\j)}
		\ifthenelse{\subscriptlabel<15}{
			\node (a\i\j) at (2*\i,-2*\j+1) {};
			\node (b\i\j) at (2*\i+.5,-2*\j+1) {};
			\node (c\i\j) at (2*\i+.5,-2*\j+1.5) {};
			\node (d\i\j) at (2*\i,-2*\j+1.5) {};
			\node (e\i\j) at (2*\i+.2,-2*\j+1.2) {};
			\node (f\i\j) at (2*\i+.7,-2*\j+1.2) {};
			\node (g\i\j) at (2*\i+.7,-2*\j+1.7) {};
			\node (h\i\j) at (2*\i+.2,-2*\j+1.7) {};
			\node at (2*\i+.3,-2*\j+.5)  {$W_{\subscriptlabel}$};
			}
			{}; 
		}
		}
		\foreach \j in {0,1,2} {
		\foreach \i in {1,...,6} {
		\pgfmathsetmacro{\subscriptlabel}{int(\i+6*\j)}
		\ifthenelse{\subscriptlabel<15}{
			\draw [dashed, blue] (a\i\j) rectangle (c\i\j);
			\draw [dashed, blue] (e\i\j) rectangle (g\i\j);
				\foreach \front/\back in {a/e,b/f,c/g,d/h}{
					\draw [dashed, blue] (\front\i\j) -- (\back\i\j);
				}
			}{};
		}
		}
		\foreach \j in {0,1,2} {
		\foreach \i in {1,...,6} {
		\foreach \m in {a,b,c,d,e,f,g,h} {
		\pgfmathsetmacro{\subscriptlabel}{int(\i+6*\j)}
		\ifthenelse{\subscriptlabel<15}{
			\filldraw[thick,draw=blue,fill=white] (\m\i\j) circle (.4mm);
			}
			{}; 
		}}}
		\foreach \j in {0,1,2} {
		\foreach \i in {1,...,6} {
		\pgfmathsetmacro{\n}{int(\i+6*\j)}
		\ifthenelse{\n<15\AND\n>1}{
			\fill[fill=red] (a\i\j) circle[radius=.75mm];
			}
			{}; 
		}
		}
		\foreach \j in {0,1,2} {
		\foreach \i in {1,...,6} {
		\pgfmathsetmacro{\n}{int(\i+6*\j)}
		\ifthenelse{\n=2\OR\n=6\OR\n=8\OR\n=11\OR\n=13\OR\n=14}{
			\fill[fill=red] (f\i\j) circle[radius=.75mm];
			}
			{}; 
		}
		}
		\foreach \j in {0,1,2} {
		\foreach \i in {1,...,6} {
		\pgfmathsetmacro{\n}{int(\i+6*\j)}
		\ifthenelse{\n=4\OR\n=5\OR\n=6\OR\n=9\OR\n=10\OR\n=11\OR\n=12}{
			\fill[fill=red] (e\i\j) circle[radius=.75mm];
			}
			{}; 
		}
		}
		\foreach \j in {0,1,2} {
		\foreach \i in {1,...,6} {
		\pgfmathsetmacro{\n}{int(\i+6*\j)}
		\ifthenelse{\n=5\OR\n=6\OR\n=9\OR\n=12}{
			\fill[fill=red] (b\i\j) circle[radius=.75mm];
			}
			{}; 
		}
		}
		\foreach \j in {0,1,2} {
		\foreach \i in {1,...,6} {
		\pgfmathsetmacro{\n}{int(\i+6*\j)}
		\ifthenelse{\n=7\OR\n=9\OR\n=10\OR\n=11\OR\n=13}{
			\fill[fill=red] (g\i\j) circle[radius=.75mm];
			}
			{}; 
		}
		}
		\fill[fill=red] (h21) circle[radius=.75mm];
		\fill[fill=red] (d61) circle[radius=.75mm];
		\fill[fill=red] (d12) circle[radius=.75mm];
		\fill[fill=red] (c22) circle[radius=.75mm];
		\fill[fill=red] (h22) circle[radius=.75mm];

		\draw[ultra thick,red] (a60) -- (b60) -- (f60) -- (e60) -- (a60);
		\draw[ultra thick,red] (a51) -- (e51) -- (f51) -- (g51);
			\foreach \start/\end in {a40/e40,a50/e50,a50/b50,a31/b31,a31/e31,a41/e41,a61/b61,a61/d61,a61/e61,a12/d12,f12/g12}{
				\draw[ultra thick, red] (\start) -- (\end);
			}
	\end{tikzpicture}
	\caption{Configurations in $Q_3$.}
	\label{fig:Q3}
\end{figure}

\begin{table}[htbp]
	\centering
		\begin{tabular}{|c|c|c|c|}
			\hline
		Configuration & Construction & Lower Bound & Upper Bound\\
			\hline\hline
		$W_1$ & $\emptyset$ & 1 & 1\\
		\hline
		$W_2$ & $A\cup B$ taking 00 mod 2 & $3/4$ & \begin{tabular}{@{}c@{}}$3/4$\\{\big(}Theorem \ref{W2}{\big)}\end{tabular}\\
		\hline
		$W_3$ & Layered: 0 mod 4 & $1/2$ & $.610043$\\
		\hline
		$W_4$ & $A\cup B$ taking $0*$ mod 3 & $(2/3)^3\approx 0.2963 $ & .$3047619$ \\
		\hline
		$W_5$ & \begin{tabular}{@{}c@{}}$A\cup B$ taking\\01,10,11 mod $(3,3)$\end{tabular} & 1/3 & .$333398$ \\
		\hline
		$W_6$ & $A\cup B$ taking $0*$ mod 2 & $4/9$ & $4/9$ (Theorem \ref{regularfamily})\\
		\hline
		$W_7$ & Layered: 0 mod 3 & $1/3$ & $.33333333366$ \\
		\hline
		$W_8$ & Layered: 0 mod 3 & $2/3$ & $.66666666673$ \\
		\hline
		$W_9$ & \begin{tabular}{@{}c@{}}$A\cup B$ taking\\00,10,21, 31 mod $(4,2)$\end{tabular} & $4/9$ & $.44444444453$ \\
		\hline
		$W_{10}$ & \makecell{Blow-up of the configuration\\ given in Figure \ref{fig:Q6config}} & \makecell{$\frac{5}{12}$\\ (Proposition \ref{prop:W10})} & $.4166666671$ \\
		\hline
		$W_{11}$ & Perfect 8-cycle blow-up & $3/8$ & \makecell{$3/8$\\ (We proved this in \cite{perfectpathpaper})}\\
		\hline
		$W_{12}$ & Layered: 0 and 1 mod 4 & $1/2$ & $.50000000004$ \\
		\hline
		$W_{13}$ & $A\cup B$ taking $0*$ mod 2 & $4/9$ & $4/9$ (Theorem \ref{regularfamily})\\
		\hline
		$W_{14}$ & Layered: 0 mod 2 & 1 & 1 \\
		\hline

		\end{tabular}
	\caption{Summary of the best results for configurations in $Q_3$.}
	\label{bestQ3s}
\end{table}

Six of these configurations are layered: the two trivial ones ($W_1$ and $W_{14}$) and four others ($W_3, W_7, W_8$, and $W_{12}$).  Letting $S$  be the layered configuration 0 mod 3 in $Q_n$ shows that $\lambda(W_7,3)\geq \frac{1}{3}$ and $\lambda(W_8,3)\geq \frac{2}{3}$.  The layered configuration 0 or 1 mod 4 shows that $\lambda(W_{12},3)\geq \frac{1}{2}$. Baber's flag algebra upper bound is within $10^{-9}$ of the lower bound for each of these.  For $W_3$ -- one vertex in $Q_3$ -- the layered configuration 0 mod 4 in $Q_n$ shows that $\lambda(W_3,3)\geq \frac{1}{2}$, while the flag algebra upper bound is .610043.

The configurations $W_6$ and $W_{13}$ are $H(3,1)$ and $H(3,2)$ respectively in the infinite family of Section \ref{sec:an_infinite_family}, so $\lambda(W_6,3)=\lambda(W_{13},3)=\frac{4}{9}$.

We showed $\lambda(W_{11},3)=\frac{3}{8}$ in \cite{perfectpathpaper} ($W_{11}$ is the induced path with 4 vertices).  The proof of Theorem \ref{W2} shows that $\lambda(W_2,3)=\frac{3}{4}$.

For $W_4$, with $[n]=A\cup B$ where $\abs{A}=\floor{\frac{2n}{3}}$ and $B=\ceil{\frac{n}{3}}$ we take $S$ to be the set of all vertices $v=(x_1,x_2,\ldots,x_n)$ where $\sum_{i\in A} x_i$ is divisible by 3.  Suppose a $Q_3$ has precisely two flip bits whose coordinate numbers are in $A$.  If the sum of all the other coordinates in $A$ is $m$, then the $Q_3$ will have configuration $W_4$ if and only if $m$ is congruent to 0 or 1 mod 3.  For example, if $m\equiv 0$ mod 3 then a vertex in the $Q_3$ will be in $S$ if and only if the two coordinates in $A$ are equal to 0, while if $m\equiv 1$ mod 3 then a vertex in the $Q_3$ will be in $S$ if and only if the two coordinates in $A$ are equal to 1 (if $m\equiv 2$ mod 3 then the configuration will be $W_{13}$).  In the limit as $n$ goes to infinity, $\frac{4}{9}$ of the sub-3-cubes will have precisely two flip bits in $A$, so $\lambda(W_4,3)\geq \frac{4}{9}\cdot\frac{2}{3}=\frac{8}{27}\approx .2963$.  Baber's flag algebra upper bound is $.304762$.

For $W_9$, we partition $[n]$ into sets $A$ and $B$ with $\abs{A}=\floor{\frac{2n}{3}}$ and $B=\ceil{\frac{n}{3}}$ and let $S$ be the set of all vertices $v=(x_1,x_2,\ldots,x_n)$ such that both $\sum_{i\in A} x_i\equiv 0$ or 1 mod 4 and $\sum_{j\in B}x_j\equiv 0$ mod 2 or both $\sum_{i\in A}x_i\equiv 2$ or 3 mod 4 and $\sum_{j\in B}x_j\equiv 1$ mod 2.  Every sub-3-cube $R$ which has two flip bits with coordinate numbers in $A$ and one in $B$ will have an exact copy of $W_9$.  For example, if $R$ is such that the sum of the non flip bit coordinates in $A$ is 1 mod 4 and the sum of the non flip bit coordinates in $B$ is 1 mod 2 then the coordinates in the flip bits for $R$ of the vertices in $S$ with the coordinate in $B$ listed third are 001, 100, 010, 110, so $R$ has an exact copy of $W_9$. In the limit as $n$ goes to infinity, $\frac{4}{9}$ of the sub-3-cubes will have precisely two flip bits with coordinate numbers in $A$, so $\lambda(W_9,3)\geq \frac{4}{9}$.

For $W_5$, we partition $[n]$ into a set $A$ of size $\ceil{\frac{n}{2}}$ and $B$ of size $\floor{\frac{n}{2}}$ and let $S$ be all vectors $v=(x_1,x_2,\ldots,x_n)$ such that $\sum_{i\in A} x_i$ and $\sum_{j\in B}x_j$ are both congruent to either 0 or 1 mod 3. If a $Q_3$ has two flip bits with coordinate numbers in $A$ and one with coordinate number in $B$, we let $y$ be the sum of the non flip bit coordinates in $A$ and $z$ be the sum of the non flip bit coordinates in $B$.  This $Q_3$ will be ``good'' if $(y,z)$ is congruent to $(0,1)$, $(0,2)$, $(2,1)$, or $(2,2)$ mod 3.  For example, if $(y,z)$ is congruent to $(0,2)$, then the coordinate in the flip bits for the $Q_3$ of the vertices in $S$ with the coordinate in $B$ listed third are 001, 011, and 101 which is an exact copy of $W_5$.  Since 4 of the 9 ordered pairs $(y,z)$ where $y,z\in\{0,1,2\}$, produce ``good'' $Q_3$'s, in the limit as $n$ goes to infinity, $\frac{4}{9}$ of the $Q_3$'s which have two flip bits in $A$ and one in $B$ will have exact copies of $W_5$.  By symmetry, $\frac{4}{9}$ of the $Q_3$'s which have one flip bit in $A$ and two in $B$ will have exact copies of $W_5$, so $\lambda(W_5,3)\geq \frac{3}{8}\cdot\frac{4}{9}+\frac{3}{8}\cdot\frac{4}{9}=\frac{1}{3}$.  Baber's flag algebra upper bound is .333398.

In the remainder of this section, we will consider the two remaining configurations in $Q_3$: $W_2$ and $W_{10}$.

The following Lemma is used in the proof that $\lambda(W_2,3)=\frac{3}{4}$.

		\begin{lem}\label{T2W2}
			Let $G$ be a graph with $n$ vertices where $n$ is even.  If $\abs{E(G)}=e$, then $G$ has at most $\min\left\{n\binom{\frac{n}{2}}{2}, \frac{e}{2}(n-2)\right\}$ induced copies of $K_{1,2}$.
			\begin{proof}
				That it has at most $n\binom{\frac{n}{2}}{2}$ was proved in \cite{PG:1975}.  The optimizing graph is $K_{\frac{n}{2},\frac{n}{2}}$.
				
				Each edge of $G$ can be in at most $n-2$ induced $K_{1,2}$s and summing over all edges $uv$ counts each $K_{1,2}$ twice.
			\end{proof}
		\end{lem}
		
Without restriction on the parity of $n$, the bound in Lemma \ref{T2W2} is $\min\left\{ \floor{\frac{n^2}{4}}\frac{n-2}{2},\frac{e}{2}(n-2) \right\}$.
		
\begin{theorem}\label{W2}
	$\ds\lambda(W_2,3)=\frac{3}{4}$.

	\begin{proof}
		We partition $[n]$ into sets $A$ and $B$ of sizes $\ceil{\frac{n}{2}}$ and $\floor{\frac{n}{2}}$ respectively.  We define a configuration $S$ in $Q_n$ by $v=(x_1,x_2,\ldots,x_n)\in S$ if and only if $\sum_{i\in A} x_i$ and $\sum_{j\in B}x_j$ are both even.  If $D$ is any sub-3-cube of $Q_n$ with one or two flip bits in coordinates with numbers in $A$ (and two or one in $B$) then $D$ has an exact copy of $W_2$.  For example, if two flip bits have coordinate numbers in $A$ and one in $B$, and if the sum of the non flip bit coordinates in $A$ is even and of the non flip bit coordinates in $B$ is odd, then the vertices of $D$ whose flip bit coordinates (with the flip bit in $B$ listed third) are 001 and 111 are in $S$, so $D$ has an exact copy of $W_2$.
		
		In the limit as $n$ goes to infinity, $\frac{3}{4}$ of the sub-3-cubes have one or two flip bits with coordinate numbers in $A$, so $\lambda(W_2,3)\geq \frac{3}{4}$.
		
		For the upper bound, suppose $\emptyset\in S$, and let $\M$ be the set of good $Q_3$s containing $\emptyset$.  Construct a graph $G_s$ with $V(G_s)=[n]$ and $E(G_s)=\{uv:\emptyset,uv$ are the vertices in $S$ for some $M\in\M\}$. If $u,v,x$ are flip bits for some $M$ in $\M$, and if $uv$ is in $M$, then neither $ux$ nor $vx$ can be in $E(G_s)$, so $\abs{\M}$ is less than or equal to the number of induced copies of the graph with three vertices and a single edge in a graph with $n$ vertices.  Equivalently, this is less than or equal to the number of induced copies of $K_{1,2}$ in a graph on $n$ vertices.  This means $\pili(W_2)\leq i(K_{1,2})=\frac{3}{4}$.
		
		Now suppose $\emptyset\not\in S$.  Let $A=\{i\in [n]:i\in S\}$, $B=[n]\setminus A$, $\abs{A}=a$, and $\abs{B}=b$.  Let $\M$ be the set of all good $Q_3$s containing $\emptyset$.  If $M\in\M$, then the two vertices of $M$ in $S$ have the structure of Type I, II, or III as shown in Figure \ref{T123W2}, when $i$ and $j$ are vertices adjacent to $\emptyset$ which are in $S$, while $u,v,x,y$ are vertices adjacent to $\emptyset$ which are not in $S$.  For example, in the Type II configuration the vertices $i$ and $ixy$ are in $S$ while $\emptyset, x,y,ix,iy,xy$ are not in $S$.
		
		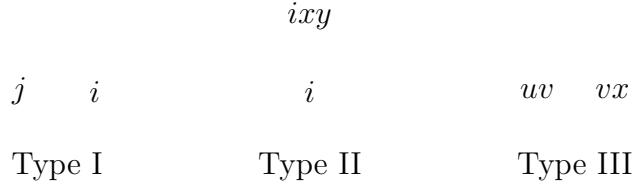
\begin{figure}[htbp]
			\centering
			\begin{tikzpicture}[grow'=up,
								sibling distance=10mm,
								level distance=10mm,
								edge from parent/.style={}]
				\node {Type I}
					child {node {$j$}}
					child {node {$i$}};
			\end{tikzpicture}
			\qquad\qquad
			\begin{tikzpicture}[grow'=up,
								sibling distance=10mm,
								level distance=10mm,
								edge from parent/.style={}]
				\node {Type II}
					child {node {$i$}
							child {node {$ixy$}}
					};
			\end{tikzpicture}
			\qquad\qquad
			\begin{tikzpicture}[grow'=up,
								sibling distance=10mm,
								level distance=10mm,
								edge from parent/.style={}]
				\node {Type III}
						child {node {$uv$}}
						child {node {$vx$}};
			\end{tikzpicture}
			\caption{The three possible structures of vertices in $S$ for $M\in\M$ where $\emptyset\not\in S$.}
			\label{T123W2}
		\end{figure}
		
		Define a graph $G_s$ by $V(G_s)=B$ and $E(G_s)=\{uv:uv$ and $vx$ are the vertices in $S$ of some Type III $M\in \M$ for some $x$ with flip bits $u,v,x\}$.  For such an $M$, $ux$ cannot be in $S$ so the number of Type III $Q_3$s in $\M$ is at most the number of induced copies of $K_{1,2}$ in $G_s$.
		
		If $L$ is a Type I $Q_3$ in $\M$ with flip bits $i,j,x$ and with $i,j$ the vertices of $L$ in $S$, then $i,j\in A$ and $x\in B$.  So the number of Type Is in $\M$ is at most $b\binom{a}{2}$.  If $L$ is a Type II $Q_3$ in $\M$ where $i,ixy$ are the vertices of $L$ in $S$, then $i\in A$ and $x,y\in B$, but $xy\not\in E(G_s)$.  So, if $e=\abs{E(G_s)}$, then the number of Type II $Q_3$s in $\M$ is at most $a\left[ \binom{b}{2}-e \right]$.  By Lemma \ref{T2W2}, we have that the number of Type III $Q_3$s in $\M$ is at most $\min\left\{b\binom{\frac{b}{2}}{2}, \frac{e}{2}(b-2)\right\}$.
		
		One good candidate to maximize $\abs{\M}$ is for $\M$ to have no Type IIIs (i.e. $\abs{E(G_s)}=0$), and for $S$ to include all vertices $j$ in $A$ and all vertices $ixy$ where $i\in A$ and $x,y\in B$.  That would give $\abs{\M}=b\binom{a}{2}+a\binom{b}{2}$.  Another good candidate would be to have $G_s$ be $K_{\frac{b}{2},\frac{b}{2}}$, so as to maximize the number of Type IIIs.  This would mean that $\abs{\M}=b\binom{a}{2} + a\left[ \binom{b}{2} - \frac{b^2}{4} \right] + b\binom{\frac{b}{2}}{2}=b\binom{a}{2} + (2a+b)\binom{\frac{b}{2}}{2}$.
		
		If $\abs{E}=e$, then we have
			\begin{equation}\tag{$\star$}\label{eqstar}
				\abs{\M}\leq b\binom{a}{2} + a\left[\binom{b}{2} - e \right] + \min\left\{b\binom{\frac{b}{2}}{2}, \frac{e}{2}(b-2)\right\}.
			\end{equation}
		Where the three summands on the right-hand side are the maximum number of ``good'' $Q_3$s of Type I, II, and III, and we have used Lemma \ref{T2W2}.
		If $e\geq \frac{b^2}{4}$ then $\min\{b\binom{\frac{b}{2}}{2},\frac{e}{2}(b-2)\}=b\binom{\frac{b}{2}}{2}$, so the right-hand side of inequality \eqref{eqstar} is a decreasing function of $e$.  Hence to maximize $\abs{\M}$ we can assume $e\leq \frac{b^2}{4}$.
		
		\textbf{\emph{\underline{Case 1:}}} If $\frac{b-2}{2}\leq a$, then
		\begin{align*}
			\abs{M} &\leq b\binom{a}{2} + a\left[ \binom{b}{2} - e \right]  + \frac{e}{2}(b-2)\\
					&\leq b\binom{a}{2} + a\left[ \binom{b}{2} - e \right] + ea\\
					&= b\binom{a}{2} + a\binom{b}{2}
		\end{align*}
		which is the size of $\M$ in the first good candidate above.
		
		\textbf{\emph{\underline{Case 2:}}} If $\frac{b-2}{2}>a$, then
		\begin{align*}
			\abs{M} &\leq b\binom{a}{2} + a\left[ \binom{b}{2} - e \right] + \frac{e}{2}(b-2)\\
					&= b\binom{a}{2} + a\binom{b}{2} + e\left( \frac{b-2}{2} - a \right)\\
					&\leq b\binom{a}{2} + a\binom{b}{2} + \frac{b^2}{4}\left( \frac{b-2}{2} - a \right)\\
					&= b\binom{a}{2} + a\binom{b}{2} + b\binom{\frac{b}{2}}{2}-\frac{b^2}{4}a\\
					&= b\binom{a}{2} + a\left( \binom{b}{2} - \frac{b^2}{4} \right) + b\binom{\frac{b}{2}}{2}\\
					&= b\binom{a}{2} + 2a\binom{\frac{b}{2}}{2} + b\binom{\frac{b}{2}}{2}\\
					&= b\binom{a}{2} + (2a+b)\binom{\frac{b}{2}}{2}
		\end{align*}
		which is the size of $\M$ in the second good candidate above.
		
		This expression can be rewritten as
		\[
			\frac{b}{2}\binom{a}{2}+\frac{b}{2}\binom{\frac{b}{2}}{2}+a\binom{\frac{b}{2}}{2}+a\binom{\frac{b}{2}}{2}+\frac{b}{2}\binom{\frac{b}{2}}{2}+\frac{b}{2}\binom{a}{2}
		\]
		which is equal to
		\begin{equation}\tag{$\star\star$}\label{eqdoublestar}
			x\binom{y}{2}+x\binom{z}{2}+y\binom{x}{2}+y\binom{z}{2}+z\binom{x}{2}+z\binom{y}{2}
		\end{equation}
		when $x=z=\frac{b}{2}$ and $y=a$.  The expression in \eqref{eqdoublestar} is the number of induced $K_{1,2}$s in a complete tripartite graph with part sizes $x,y$, and $z$.  We know that $K_{\frac{n}{2},\frac{n}{2}}$ is the graph with $n$ vertices which has the maximum number of induced $K_{1,2}$s (Remark \ref{K12}. in Section \ref{sec:inducibility}), so \eqref{eqdoublestar} attains its maximum value when $x=z=\frac{n}{2}$ and $y=0$, so $b=n$ and $a=0$.  The size of $\M$ for the first candidate $a\binom{b}{2}+b\binom{a}{2}$ is the value of \eqref{eqdoublestar} when $x=a$, $y=b$, and $z=0$, so it attains its maximum value when $a=b=\frac{n}{2}$ and both good candidates have size
		\[
			2\cdot \frac{n}{2}\binom{\frac{n}{2}}{2}=\frac{n^2(n-2)}{8}=\frac{3}{4}\binom{n}{3}\frac{n}{n-1}
		\]
		and $\abs{\M}$ cannot be bigger.  Hence
		\[
			\frac{3}{4}\leq\lambda(W_2,3)\leq \pil(W_2,3)\leq\frac{3}{4}.
		\]
	\end{proof}
\end{theorem}

We remark that in the construction we have with density $\frac{3}{4}$, of the vertices not in $S$, $\frac{2}{3}$ of them are in good $Q_3$'s only of the type of the first good candidate (those vertices which have an odd sum in precisely one of $A$ or $B$) and $\frac{1}{3}$ are in good $Q_3$'s only of the type of the second candidate (those vertices with an odd sum in both $A$ and $B$).  The local density at all vertices is $\frac{3}{4}$.

We suspect that no nontrivial $d$-cube density can be bigger than $\frac{3}{4}$ and that this is the only one bigger than $\frac{2}{3}$.

\begin{conj}
	If $H$ is a configuration in $Q_d$ such that $\frac{2}{3}<\ex(H,d)<1$ then $d=3$ and $H$ is an exact copy of $W_2$.
\end{conj}

We will use a \emph{blow-up} (introduced in \cite{perfectpathpaper}) of the miraculous configuration $T$ in $Q_6$ shown in Figure \ref{fig:Q6config} to get a lower bound for $\lambda(W_{10},3)$. Let $H$ be a configuration in $Q_d$.  We say the configuration $S$ in $Q_n$ is a blow-up of $H$ if it is obtained as follows.  We partition $[n]$ into parts $A_1,A_2,\ldots,A_d$.  For each $v=(x_1,x_2,\ldots,x_n)$ in $V_n$ we define a vector $u=(y_1,y_2,\ldots,y_d)$ in $V_d$ by $y_j\equiv \sum_{i\in A_j}x_i$ mod 2, and put $v$ in $S$ if and only if $u\in H$.  Any sub-$d$-cube of $Q_n$ which has one flip bit in each $A_j$ will have an exact copy of $H$ (and perhaps others will as well).  Using an equipartition of $[n]$ into $d$ parts shows that $\lambda(H,d)\geq \frac{d!}{d^d}$ for each $H$ in $Q_d$ (Proposition 6 of \cite{perfectpathpaper}).

\begin{prop}\label{prop:W10}
	$\ex(W_{10},3)\geq \frac{5}{12}$
	\begin{proof}
		Let $S$ be a blow-up in $Q_n$ of the configuration $T$ in $Q_6$ shown in Figure \ref{fig:Q6config}.  If $n$ is large and $[n]$ is equipartitioned into 6 parts, then the probability that a randomly chosen sub-3-cube has its 3 flip bits in different parts is $\frac{6}{6}\cdot\frac{5}{6}\cdot\frac{4}{6} = \frac{5}{9}$.  To complete the proof we need to show that $\frac{3}{4}$ of the $\binom{6}{3}\cdot 2^3=160$ sub-3-cubes of $Q_6$ have an intersection with $T$ which is an exact copy of $W_{10}$, since $\frac{5}{9}\cdot\frac{3}{4}=\frac{5}{12}$.
		
		\begin{figure}[htbp]
			\centering
			\begin{tikzpicture}[grow'=up,
								sibling distance=30mm,
								level distance=10mm,
								edge from parent/.style={}]
				\node {$\emptyset$}
					child {node {$1$}
						child {node {}
							child{node {124 146 163 135 152}
								child{node {1234 1456 1623 1345 1562}
									 }
								 }
								}
						}
					child {node {}
						child {node{}
							child{node{}
								child {node {}
									child {node {}
										child {node {123456}}
										}
									}
								}
							}
						}
					child {node {}
						child {node{23 34 45 56 62}
							child {node {235 346 452 563 624}
								child {node {}
									child {node {23456}}
									}
								}
							}
						};
			\end{tikzpicture}
			\caption{The configuration $T$ in $Q_6$ used for Proposition \ref{prop:W10}}
			\label{fig:Q6config}
		\end{figure}
		
		It is not hard to check that for each $u,v$ in $T$, there are 10 automorphisms of $Q_6$ which map $u$ to $v$ and $T$ to itself.  For example, if $\psi$ is such an automorphism such that $\psi(\emptyset)=\emptyset$, then $\psi(1)=1$, $\psi(2)=j$ where $j\in\{$2,3,4,5,6$\}$, $\psi(3)=j+1$ ($\psi(3)=2$ if $j=6$) or $\psi(3)=j-1$ ($\psi(3)=6$ if $j=2$), and that completely determines $\psi$.  That means the group of automorphisms of $Q_6$ which fix the set $T$ is vertex transitive within $T$ and outside of $T$.  Hence it suffices to show that 15 of the 20 sub-3-cubes containing a vertex in $T$ (say $\emptyset$) have exact copies of $W_{10}$ and that 15 of the 20 sub-3-cubes containing a vertex not in $T$ (say 2) have exact copies of $W_{10}$.
		
		Of the 20 sub-3-cubes containing $\emptyset$, only the ones containing $234,345,456,562$, and $623$ do not have exact copies of $W_{10}$.  For example, the one with flip bits 2,3, and 5 contains the vertices $\emptyset$, 23, and 235 of $T$, the one with flip bits 1,2, and 3 contains the vertices $\emptyset$, 1, and 23 of $T$, and the one with flip bits 1, 2, and 4 contains the vertices $\emptyset$, 1, and 124 of $T$.
		
		Of the 20 sub-3-cubes containing 2, the five which do not have exact copies of $W_{10}$ are the ones with flip bits 1,3,6;\ 1,4,5;\ 2,3,4;\ 2,3,6;\ 2,5,6.  For example, the one with flip bits 1,3,5 contains the vertices 23,125, and 235 of $T$ and the one with flip bits 2,3,5 contains the vertices $\emptyset$, 23, and 235 of $T$.
	\end{proof}
\end{prop}
		
		We remark that the generic lower bound $\ex(H,d)\geq \frac{d!}{d^d}$ is attained by $Z_3$ in $Q_2$ and by the perfect 8-cycle $C_4$ in $Q_4$, but not by any configuration in $Q_3$, since $\ex(H,3)\geq \frac{8}{27}$ for all configurations $H$ in $Q_3$.




\section{Configurations in $Q_4$} 
\label{sec:configurations_in_q_4}

 In \cite{perfectpathpaper} we determined the $d$-cube-density of the perfect 8-cycle $C_8$ in $Q_4$.  In this section, we will determine the $d$-cube-density for two other configurations in $Q_4$.

\begin{theorem}\label{38inQ4}
	If $Y$ is the configuration $\{0000,1100,0011,1111\}$ in $Q_4$ (see Figure \ref{Q4config}), then $\ds\lambda(Y,4)=\frac{3}{8}$.
	
	\begin{figure}[htbp]
		\centering
		\begin{tikzpicture}[scale=2,inner sep=.5mm]
			\foreach \i in {1,2} {
				\node (a\i) at (\i,1+\i/2.5) {};
				\node (b\i) at (\i+.5,1+\i/2.5) {};
				\node (c\i) at (\i+.5,1.5+\i/2.5) {};
				\node (d\i) at (\i,1.5+\i/2.5) {};
				\node (e\i) at (\i+.2,1.2+\i/2.5) {};
				\node (f\i) at (\i+.7,1.2+\i/2.5) {};
				\node (g\i) at (\i+.7,1.7+\i/2.5) {};
				\node (h\i) at (\i+.2,1.7+\i/2.5) {};
			}
			\foreach \i in {1,2} {
				\draw [thick, blue] (a\i) rectangle (c\i);
				\draw [thick, blue] (e\i) rectangle (g\i);
					\foreach \front/\back in {a/e,b/f,c/g,d/h}{
						\draw [thick, blue] (\front\i) -- (\back\i);
					}
					\foreach \letter in {a,b,c,d,e,f,g,h}{
						\draw [dashed,blue] (\letter1) -- (\letter2);
					}
			}
			\foreach \i in {1,2} {
			\foreach \j in {a,b,c,d,e,f,g,h} {
			\filldraw[thick,draw=blue,fill=white] (\j\i) circle (.3mm);
			}}
			\fill[fill=red] (a1) circle[radius=.75mm];
			\fill[fill=red] (c1) circle[radius=.75mm];
			\fill[fill=red] (e2) circle[radius=.75mm];
			\fill[fill=red] (g2) circle[radius=.75mm];
		\end{tikzpicture}
		\caption{The configuration $Y$.}
		\label{Q4config}
	\end{figure}
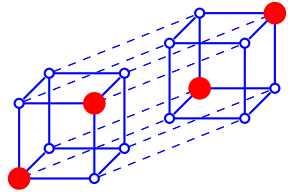

	\begin{proof}
		First we give a construction to show $\ex(Y,4)\geq \frac{3}{8}$. Partition $[n]$ into sets $A$ and $B$ and let $S$ be the set of all vertices in $Q_n$ given by binary $n$-tuples with even weight in both $A$ and $B$.  This gives a ``good'' $Q_4$ for each $Q_4$ with two flip bits in each of $A$ and $B$.  If it is an equipartition and $n$ is large then $\frac{3}{8}$ of the $Q_4$s are ``good''.
		
		Suppose $\emptyset\in S$ and let $\M$ be the set of good $Q_4$s containing $\emptyset$. We construct a graph $G_s$ with $V(G_s)=[n]$ and $E(G_s)=\{uv:\emptyset,uv,xy,uvxy$ are the vertices in $S$ of some $M\in\M \}$.  If $uv$ and $xy$ are in $M\in\M$, then neither $ux,uy,vx$, nor $vy$ can be in $E(G_s)$, so $\abs{\M}$ is less than or equal to the number of induced copies of $2K_2$ in $G_s$.  That means $\pili(Y)\leq i(2K_2)=\ds\frac{3}{8}$.
		
		Now suppose $\emptyset\not\in S$.  Let $A=\{i\in[n]:i\in S\}$, $B=[n]\setminus A$, $\abs{A}=a$, and $\abs{B}=b$.  Let $\M$ be the set of all good $Q_4$s containing $\emptyset$.  If $M\in\M$, then the four vertices of $M$ in $S$ have the structure of Type I or Type II in Figure \ref{T12Q4}, where $i,j,u,v,x,y\in[n]$ with $i$ and $j$ in $S$, but $u,v,x,$ and $y$ not in $S$.
		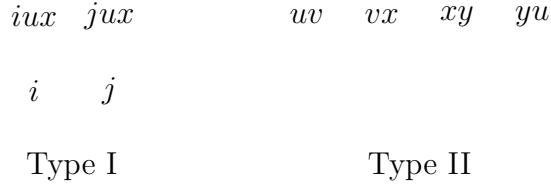
\begin{figure}[htbp]
			\centering
			\begin{tikzpicture}[parent anchor=east,child anchor=west,grow=north,
				sibling distance=10mm, level distance=10mm,
				every node/.style={text=black},
				edge from parent/.style={}]
				\node {Type I}
					child {node {$j$}
						child {node {$jux$}}
						}
					child {node {$i$}
						child {node {$iux$}}
						};
			\end{tikzpicture}
			\qquad\qquad
			\begin{tikzpicture}[grow=north,
				sibling distance=10mm, level distance=10mm,
				every node/.style={text=black},
				edge from parent/.style={}]
				\node {Type II}
					child {node {}
						child {node {$yu$}}
						child {node {$xy$}}
						child {node {$vx$}}
						child {node {$uv$}}
					};
			\end{tikzpicture}
			\caption{The two possible structures of vertices in $S$ for $M\in\M$ where $\emptyset\not\in S$.}
			\label{T12Q4}
		\end{figure}
		
		Define a graph $G_s$ by $V(G_s)=B$ and $E(G_s)=\{uv:uv,vx,xy,yu$ are the vertices in $S$ of some Type II $M\in\M$ for some $x,y$ in $[n]\}$.  For such an $M$, neither $ux$, nor $vy$ can be in $S$, so the number of Type II $Q_4$s in $\M$ is at most the number of induced copies of $K_{2,2}$ in $G_s$.
		
		\begin{lem}\label{T2Q4}
			Let $G$ be a graph with $n$ vertices where $n$ is even.  If $\abs{E(G)}=e$, then $G$ has at most $\min\left\{\binom{\frac{n}{2}}{2}^2,\frac{e}{4}\frac{(n-2)^2}{4}\right\}$ induced copies of $K_{2,2}$.
			\begin{proof}
				That it has at most $\binom{\frac{n}{2}}{2}^2$ copies of $K_{2,2}$ is proved in \cite{BS:1994} and \cite{Bollobas:1986bfa} (the optimizing graph is $K_{\frac{n}{2},\frac{n}{2}}$).  If $uv\in E(G)$, define an auxiliary graph $F$ with $V(F)=V(G)\setminus\{u,v\}$ and $E(F)=\{xy:\{u,v,x,y\}$ induces $K_{2,2}\}$.  The graph $F$ is triangle-free since if $\{u,v,x,y\}$ and $\{u,v,x,z\}$ both induce $K_{2,2}$, then either $\{uy,uz\}\subseteq E(G)$ or $\{vy,vz\}\subseteq E(G)$.  In either case, $\{u,v,y,z\}$ induces $K_{1,3}$.  Since $F$ is triangle free, by Tur{\'a}n's theorem, $uv$ is in at most $\frac{(n-2)^2}{4}$ induced $K_{2,2}$s.  Finally, summing over all edges $uv$ counts each $K_{2,2}$ four times.
			\end{proof}
		\end{lem}
		
		If $L$ is a good Type I in $\M$ where $i,j,iux$, and $jux$ are the vertices of $L$ in $S$, then $i,j\in A$, $u,x\in B$, but $ux\not\in E(G_s)$.  If $\abs{E(G_s)}=e$, then  the number of Type I $Q_4$s in $\M$ is at most $\left[\binom{b}{2}-e\right]\binom{a}{2}$ and of Type II, by Lemma \ref{T2Q4}, is at most $\min\left\{\binom{\frac{b}{2}}{2}^2,\frac{e}{16}(b-2)^2\right\}$ (with slight modification if $b$ is odd).  If $a$ and $b$ are fixed, then one good candidate to maximize $\abs{\M}$ is for $\M$ to have no Type II $Q_4$s.  Then $G_s$ has no edges and $\abs{M}=\binom{a}{2}\binom{b}{2}$.  Another good candidate is when $\M$ has the maximum possible number of Type II $Q_4$s.  Then $G_s$ is $K_{\frac{b}{2},\frac{b}{2}}$ (assuming $b$ is even) and $\abs{\M} = \left[ \binom{b}{2} - \frac{b^2}{4} \right] \binom{a}{2} + \binom{\frac{b}{2}}{2}^2 = \frac{b(b-2)}{4} \binom{a}{2} + \binom{\frac{b}{2}}{2}^2$.
		
		If $e=E(G_s)$, we have $\abs{\M} \leq \left[ \binom{b}{2} - e \right] \binom{a}{2} + \min\left\{ \binom{\frac{b}{2}}{2}^2, \frac{e}{4} \frac{(b-2)^2}{4} \right\}$.
		
		\textbf{\emph{\underline{Case 1:}}} If $e\geq \frac{b^2}{4}$, then
		\begin{align*}
			\abs{\M} &\leq \left[ \binom{b}{2}-\frac{b^2}{4} \right] \binom{a}{2} + \binom{\frac{b}{2}}{2}^2\\
					&= \frac{b(b-2)}{4} \binom{a}{2} + \binom{\frac{b}{2}}{2}^2
		\end{align*}
		which is the size of $\M$ in the second good candidate above.
		
		\textbf{\emph{\underline{Case 2:}}} If $e<\frac{b^2}{4}$, then
		\begin{align*}
			\abs{\M} &\leq \left[ \binom{b}{2} - e \right] \binom{a}{2} + \frac{e}{16}(b-2)^2\\
					 &= \binom{a}{2} \binom{b}{2} + e \left[ \frac{1}{16}(b-2)^2 - \binom{a}{2} \right].
		\end{align*}
		
		If $\frac{1}{16}(b-2)^2\leq \binom{a}{2}$, then $\abs{\M}\leq \binom{a}{2} \binom{b}{2}$ which is the size of $\M$ in the first good candidate above.
		
		If $\frac{1}{16}(b-2)^2 > \binom{a}{2}$, then
		\begin{align*}\label{upper}
			\abs{\M} &< \binom{a}{2} \binom{b}{2} + \frac{b^2}{4} \left[ \frac{1}{16}(b-2)^2 - \binom{a}{2} \right]\\
					 &= \binom{a}{2} \left[ \binom{b}{2} - \frac{b^2}{4} \right] + \left(\frac{b(b-2)}{8}\right)^2\\
					 &= \frac{b(b-2)}{4} \binom{a}{2} + \binom{\frac{b}{2}}{2}^2 \tag{$*$}
		\end{align*}
		
		the same upper bound as in Case 1.
		
		Clearly the maximum value of $\binom{a}{2}\binom{b}{2}$ is $\binom{\frac{n}{2}}{2}\binom{\frac{n}{2}}{2} = \frac{n^2(n-2)^2}{64} = \frac{3}{8}\binom{n}{4}\frac{n(n-2)}{(n-1)(n-3)}$ (with a slight modification if $n$ is odd).
		
		\begin{lem}
			If $n$ is even and $x, y$, and $z$ are nonnegative integers such that $x+y+z=n$, then the maximum value of
			\begin{equation*}\tag{$**$}\label{binsum}
				\binom{x}{2}\binom{y}{2}+\binom{x}{2}\binom{z}{2}+\binom{y}{2}\binom{z}{2}
			\end{equation*}
			is $\binom{\frac{n}{2}}{2}^2$.
			\begin{proof}
				This function counts the number of induced copies of $K_{2,2}$ in a complete tripartite graph with parts $X,Y,$ and $Z$ with part sizes $x,y,$ and $z$, respectively, subject to the constraint $x+y+z=n$.  In \cite{BS:1994}, it was shown that for all $n\geq 4$ the maximum number of induced copies of $K_{2,2}$ in any graph is $\binom{\frac{n}{2}}{2}^2$.
			\end{proof}
		\end{lem}

		If $x=a$ and $y=z=\frac{b}{2}$, then \eqref{binsum} reduces to \eqref{upper}, so the maximum of \eqref{upper} occurs when $a=0$ and $b=n$ and is equal to $\frac{3}{8}\binom{n}{4}\frac{n(n-2)}{(n-1)(n-3)}$.  Hence, $\frac{3}{8}$ is an upper bound for $\pilo(Y)$ and $\pili(Y)$, so $\frac{3}{8}\leq\lambda(Y,4)\leq\pil(Y,4)\leq\frac{3}{8}$.
		
	\end{proof}
\end{theorem}

\begin{theorem}\label{thm:halfdensity}
	If $Z$ is the configuration $\{ 0000,1100,1010,0110 \}$ in $Q_4$ (see Figure \ref{fig:halfdensity}), then $\lambda(Z,4)=\frac{1}{2}$.
	\begin{figure}[htbp]
		\centering
		\begin{tikzpicture}[scale=2,inner sep=.5mm]
			\foreach \i in {1,2} {
				\node (a\i) at (\i,1+\i/2.5) {};
				\node (b\i) at (\i+.5,1+\i/2.5) {};
				\node (c\i) at (\i+.5,1.5+\i/2.5) {};
				\node (d\i) at (\i,1.5+\i/2.5) {};
				\node (e\i) at (\i+.2,1.2+\i/2.5) {};
				\node (f\i) at (\i+.7,1.2+\i/2.5) {};
				\node (g\i) at (\i+.7,1.7+\i/2.5) {};
				\node (h\i) at (\i+.2,1.7+\i/2.5) {};
			}
			\foreach \i in {1,2} {
				\draw [thick, blue] (a\i) rectangle (c\i);
				\draw [thick, blue] (e\i) rectangle (g\i);
					\foreach \front/\back in {a/e,b/f,c/g,d/h}{
						\draw [thick, blue] (\front\i) -- (\back\i);
					}
					\foreach \letter in {a,b,c,d,e,f,g,h}{
						\draw [dashed,blue] (\letter1) -- (\letter2);
					}
			}
			\foreach \i in {1,2} {
			\foreach \j in {a,b,c,d,e,f,g,h} {
			\filldraw[thick,draw=blue,fill=white] (\j\i) circle (.3mm);
			}}
			\fill[fill=red] (a1) circle[radius=.75mm];
			\fill[fill=red] (c1) circle[radius=.75mm];
			\fill[fill=red] (f1) circle[radius=.75mm];
			\fill[fill=red] (h1) circle[radius=.75mm];
		\end{tikzpicture}
		\caption{The configuration $Z$ for Theorem \ref{thm:halfdensity}.}
		\label{fig:halfdensity}
	\end{figure}
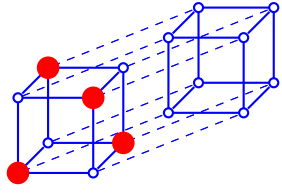
	\begin{proof}
		The construction to show $\lambda(Z,4)\geq \frac{1}{2}$ is similar to the one for $Y$ in Theorem \ref{38inQ4}.  Partition $[n]$ into sets $A$ and $B$ of size $\ceil{\frac{n}{2}}$ and $\floor{\frac{n}{2}}$ and let $S$ be the set of all veertices in $Q_n$ given by binary $n$-tuples with even weight in both $A$ and $B$.  This gives a good $Q_4$ for each $Q_4$ with three flip bits in $A$ and one in $B$ or one flip bit in $A$ and three in $B$. In the limit as $n$ goes to infinity, $\frac{4}{16}+\frac{4}{16}=\frac{1}{2}$ of the $Q_4$'s will be ``good'', so $\lambda(H,4)\geq \frac{1}{2}$.
		
		For the upper bound, suppose $\emptyset\in S$ and let $\M$ be the set of good $Q_4$s containing $\emptyset$.  We define a graph $G_S$ with $V(G_S)=[n]$ and $E(G_S)=\{xy : \emptyset,xy,yz,xz$ are the vertices in $S$ of some $M\in\M$ for some $z\in[n]\}$.  If $x,y,z,w$ are the coordinates of a good $Q_4$ where $\emptyset,xy,yz,xz$ are the vertices in $S$, then $wx,wy,wz$ are not in $E(G_S)$, so $\{w,x,y,z\}$ induces $K_3$ plus an isolated vertex in $G_S$.  Since this is the complement of $K_{1,3}$, $\pilin(H,4)=i(K_{1,3})=\frac{1}{2}$.
		
		Now suppose $\emptyset\not\in S$.  Let $A=\{i\in [n]:i\in S\}$, $B=[n]\setminus A$, $\abs{A}=a$, $\abs{B}=b$.  Let $\M$ be the set of all good $Q_4$s containing $\emptyset$.  If $M\in\M$ then the four vertices of $M$ in $S$ have the structure of Type I, Type II, or Type III in Figure \ref{fig:halfdensitytypes} (where $i,j,k\in A$ and $w,x,y,z\in B$).
		
		\begin{figure}[htbp]
			\centering
			\begin{tikzpicture}[grow'=up,
								sibling distance=10mm,
								level distance=10mm,
								edge from parent/.style={}]
				\node {Type I}
					child {node {$i$}}
					child {node {$j$}
						child {node {$ijk$}}
					}
						child {node {$k$}};
			\end{tikzpicture}
			\qquad\qquad
			\begin{tikzpicture}[grow'=up,
								sibling distance=10mm,
								level distance=10mm,
								edge from parent/.style={}]
				\node {Type II}
					child {node {$i$}
							child {node {$ixy$}}
							child {node {$iyz$}}
							child {node {$ixz$}}
					};
			\end{tikzpicture}
			\qquad\qquad
			\begin{tikzpicture}[grow'=up,
								sibling distance=10mm,
								level distance=10mm,
								edge from parent/.style={}]
				\node {Type III}
					child {node {$wx$}}
					child {node {$wy$}
						child {node {$wxyz$}}
					}
					child {node {$wz$}};
			\end{tikzpicture}
			\caption{The three structures of vertices in $S$ for $M\in\M$ where $\emptyset\not\in S$.}
			\label{fig:halfdensitytypes}
		\end{figure}
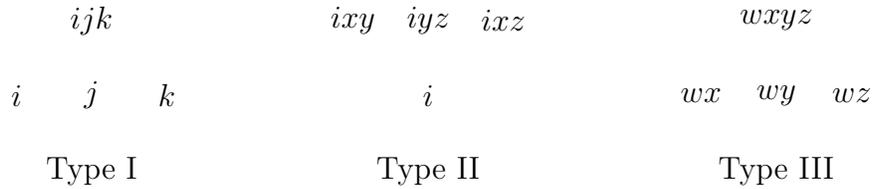
		
		Define a graph $G$ by $V(G)=A\cup B$ and $E(G)=\{ix:i\in A\textrm{ and }\ab x\in B\} \cup \ab \{wx: wx,\ab wy, \ab wz, \ab wxyz$ are the vertices in $S$ of a Type III $M\in\M$ for some $y,z\in B\}$. If $M$ is a Type I $Q_4$ with coordinates $i,j,k,x$, then $\{i,j,k,x\}$ induces $K_{1,3}$ in $G$.  If $M$ is a Type III $Q_4$ with vertices $wx,wy,wz,wxyz$ in $S$, then $\{w,x,y,z\}$ induces $K_{1,3}$ in $G$, since $xy,yz$ and $xz$ are not edges in $G$.  That means $\{i,x,y,z\}$ induces $K_{1,3}$ in $G$ since $ix,iy,iz$ are all edges.  It also means that the number of Type II $Q_4$s in $\M$ is at most the number of $K_{1,3}$s in $G$ with one vertex in $A$ and three vertices in $B$, since if $i,x,y,z$ are the coordinates of a Type II M, then $xy,yz,$ and $xz$ are all non-edges.  Thus $\abs{\M}$ is at most the number of $K_{1,3}$s in $G$ which have precisely 3,1, or 0 vertices in $A$, so is certainly at most the maximum number of $K_{1,3}$s in a graph with $n$ vertices.  Hence $\pilout(Z,4)\leq i(K_{1,3})=\frac{1}{2}$, and $\frac{1}{2}\leq \lambda(Z,4)\leq\pil(Z,4)\leq \frac{1}{2}$.
	\end{proof}
\end{theorem}

We remark that since the only optimizing host graph which maximizes the number of induced $K_{1,3}$ subgraphs is complete bipartite, the graph $G$ defined above can only be optimal if either there are no Type III $M\in\M$ (so both $A$ and $B$ are independent sets), or $A=\emptyset$, each $M\in \M$ is Type III, and $B$ induces a complete bipartite graph with parts asymptotically, but not exactly equal in size (Remark \ref{K13} in Section \ref{sec:inducibility}).


\section{Open Problems and conjectures} 
\label{sec:open_problems_and_conjectures}

In this section, we will restate our conjectures stated earlier in the paper, state some other ones, and discuss some other problems for possible further research.

\subsection{Conjectures on $d$-cube density}
Baber's flag algebra upper bounds are within $10^{-9}$ of our lower bounds for five of the configurations in $Q_3$, so we conjecture equality holds for all of them.

\begin{conj}
	The lower bounds for $W_7,W_8,W_9,W_{10}$, and $W_{12}$ given in Table \ref{bestQ3s} are, in fact, the exact 3-cube densities for these configurations.
\end{conj}

It follows from Theorem \ref{thm:local_trivial_layered} that if a configuration $H$ in $Q_d$ is not layered, then $\lambda(H,d)<1$.  It is easy to show that the only layered configurations with $d$-cube density equal to 1 are the trivial ones and that for each $d\geq 3$, there is a layered configuration in $Q_d$ with $d$-cube density equal to $\frac{2}{3}$.  We have an example of a configuration with 5-cube density equal to at least the inducibility of $K_{2,3}$, which is $\frac{5}{8}$ (configuration $E(5,2)$ in Section \ref{subsec:another_infinite_family}).

\begin{conj}
	If $H$ is a configuration in $Q_d$ such that $\frac{5}{8}<\lambda(H,d)<1$ then either
	\begin{enumerate}
		\item $H$ is layered and $\lambda(H,d)=\frac{2}{3}$ or
		\item $d=3$, $H=W_2$, and $\lambda(H,d)=\frac{3}{4}$.
	\end{enumerate}
\end{conj}

\subsection{Another Inifinite Family}
\label{subsec:another_infinite_family}

Let $d$ and $i$ be positive integers with $1\leq i<d$.  We define the configuration $E(d,i)$ in $Q_d$ by
\[
	E(d,i)=\left\{ (x_1,x_2,\ldots,x_d)\in V_d {\bigg\vert} \sum_{j=1}^i x_j \textrm{ and }\sum_{j=i+1}^d x_j \textrm{ are both even} \right\}.
\]
So two weights must be even for each vertex in $E(d,i)$, whereas only one must be even for each vertex in the configuration $H(d,i)$ of Section \ref{sec:an_infinite_family}.  Note that $E(d,d-i)$ is an exact copy of $E(d,i)$ for all $i$ and $d$.  We remark that $E(3,1)$ is the configuration $W_2$ in $Q_3$ and $E(4,1)$ and $E(4,2)$ are the configurations $Z$ and $Y$ in $Q_4$ of Section \ref{sec:configurations_in_q_4}.

To get a lower bound for $\lambda(E(d,i),d)$ we let $x$ be a real number in $\left(0,\frac{1}{2}\right]$, $m=\ceil{xn}\in\left[1,\frac{n}{2}\right]$, and we define a configuration $S_x$ in $Q_n$ by $S_x=$ $\{(x_1,x_2,\ldots,x_n)\in V_n : \sum_{j=1}^m x_j$ and $\sum_{j=m+1}^n x_j$ are both even$\}$.  Then any sub-$d$-cube of $Q_n$ with precisely $i$ or $d-i$ flip bits in $[1,m]$ is ``good'', and in the limit as $n$ goes to infinity this is a fraction $f_{d,i}(x)=\binom{d}{i}\left[ x^i(1-x)^{d-i} + x^{d-i}(1-x)^i \right]$ of all sub-$d$-cubes.  The function $f_{d,i}(x)$ is also the fraction of subsets of $d$ vertices of $K_{m,n-m}$ which induce $K_{i,d-i}$.  As shown in \cite{BS:1994}, $i(K_{i,d-i})$ is the maximum value of $f_{d,i}(x)$ where $x\in\left( 0,\frac{1}{2} \right]$, so we have shown the following.

\begin{prop}\label{Eprop}
	For all integers $i$ and $d$ with $1\leq i<d$, $\ex(E(d,i),d)\geq i(K_{i,d-i})$.
\end{prop}

In \cite{BS:1994}, Brown and Sidorenko determined for which complete bipartite graphs $K_{s,s+t}$ the optimizing host complete bipartite graph on $n$ vertices, as $n$ goes to infinity, can be taken to be asymptotically equibipartite.  It follows from their work that for $i\in\left[ 1,\frac{d}{2} \right]$, $f_{d,i}(x)$ is maximized when $x=\frac{1}{2}$ if and only if $i\geq \frac{d-\sqrt{d}}{2}$.  So $\lambda(E(4,1),4)\geq f_{4,1}\left(\frac{1}{2}\right) = i(K_{1,3}) = \frac{1}{2}$ and $\ex(E(4,2),4) \geq f_{4,2}\left(\frac{1}{2}\right) = i(K_{2,2}) = \frac{3}{8}$ (and Theorems \ref{38inQ4} and \ref{thm:halfdensity} say that equality holds).

However, $f_{5,1}(x)$ is maximized when $x=\frac{3-\sqrt{3}}{6}$, so $\ex(E(5,1),5) \geq f_{5,1}\left( \frac{3-\sqrt{3}}{6} \right) = i(K_{1,4}) = \frac{5}{12}$.  We believe equality holds.

\begin{conj}
	Equality holds in Proposition \ref{Eprop}.
\end{conj}


\subsection{One Vertex in $Q_d$}
\label{subsec:one_vertex_in_q_d}

Let $U_d$ be the configuration in $Q_d$ consisting of a single vertex.  We have been unable to determine $\lambda(U_d,d)$ for any $d\geq 2$.  Since $U_d$ is a layered configuration (with weight set $\{0\}$), it makes sense to use a layered configuration $S$ in $Q_n$ to get a construction for a lower bound.  Letting $S$ be the layered configuration 0 mod $(d+1)$ in $Q_n$ shows that $\lambda(U_d,d)\geq\frac{2}{d+1}$.  This is the best lower bound we have for $\lambda(U_2,2)$ ($Z_2$ in Table \ref{bestQ2s}) and $\lambda(U_3,3)$ ($W_3$ in Table \ref{bestQ3s}).  The flag algebra upper bounds are somewhat larger: $\frac{2}{3}\leq \ex(U_2,2)\leq .685714$ and $\frac{1}{2}\leq \ex(U_3,3)\leq .610043$.

A Hamming perfect code can be used to construct a non-layered configuration $S$ in $Q_n$ which produces a lower bound for $\ex(U_3,3)$ which is almost as good.  Let $H$ be the configuration in $Q_7$ consisting of the 16 vectors in the length 7 dimension 4 Hamming perfect code. (One relization of it is $\emptyset$ (the 0 vector), the 7 incidence vectors of a Fano plane, say the 7 cyclic permutations of 124, and the complements of these 8 vectors).  The minimum distance in this code is 3, and each vector is Hamming distance 3 from 7 other vectors.  So the number of sub-3-cubes which contain two vertices of $H$ is $\frac{1}{2}\cdot 16\cdot 7=56$.  Since the average number of vertices of $H$ in a sub-3-cube is $\frac{2^4}{2^7}\cdot 8 =1$, there must be 56 sub-3-cubes with no vertices of $H$.  Hence there are $\binom{7}{3}\cdot 2^4-112=448$ sub-3-cubes which have an exact copy of $U_3$, a fraction $\frac{448}{560}=\frac{4}{5}$ of the sub-3-cubes of $Q_7$.  If $n$ is large and $S$ is a blow up of $H$ in $Q_n$ (equipartition of $[n]$ into 7 parts) then a sub-3-cube of $Q_n$ with flip bits in different parts will have an exact copy of a sub-3-cube of the $Q_7$ (with configuration $H$), and $\frac{4}{5}$ of these will have an exact copy of $U_3$.  Hence $\lambda(U_3,3)\geq \frac{7}{7}\cdot\frac{6}{7}\cdot\frac{5}{7}\cdot\frac{4}{5}=\frac{24}{49}$, not quite as good as the lower bound $\frac{1}{2}$ using the layered configuration 0 mod 4 in $Q_n$.

Another way to get a lower bound for $\ex(U_d,d)$ is to let $R$ be the configuration in $Q_n$ obtained by the following random process: for each $v\in V_n$, put $v\in R$ with (uniform independent) probability $\frac{1}{2^d}$.  The probability that a sub-$d$-cube has an exact copy of $U_d$ is $2^d\cdot\frac{1}{2^d}\cdot\left(\frac{2^d-1}{2^d}\right)^{2^d-1}$ so $\lambda(U_d,d)\geq\left( \frac{2^d-1}{2^d} \right)^{2^d-1}$.  Of course this is larger than $\frac{1}{e}$ and, as noted in Section \ref{sec:an_infinite_family}, is equal to $\frac{1}{e}$ in the limit as $d$ goes to infinity.  This lower bound is larger than $\frac{2}{d+1}$ if $d>4$.

\begin{conj}
	If $d\geq 5$ then $\lambda(U_d,d)=\left( \frac{2^d-1}{2^d} \right)^{2^d-1}$.
\end{conj}

This has the same flavor as a special case of the edge-statistics conjecture of Alon et al. \cite{Alon:2020} which says (though formulated differently) that the limit as $k$ goes to infinity of the inducibility of a graph with $k$ vertices and one edge is $\frac{1}{e}$.


\bibliographystyle{amsplain}
\bibliography{combinedbib}
\end{document}